\documentclass[11pt]{article}
\usepackage{amsfonts,amsmath,amsthm, amssymb}
\usepackage{latexsym, euscript}
\newtheorem{lemma}{Lemma}[section]
\newtheorem{theorem}[lemma]{Theorem}
\newtheorem{definition}[lemma]{Definition}
\newtheorem{proposition}[lemma]{Proposition}
\newtheorem{corollary}[lemma]{Corollary}

\newtheorem{remark}[lemma]{Remark}
\newtheorem{example}[lemma]{Example}

\begin{document}

\begin{center}
{\LARGE Observable Optimal State Points of Subadditive Potentials }

\vspace{.3cm}
Eleonora Catsigeras\footnote {Instituto de Matem\'{a}tica, Universidad de la
Rep\'{u}blica, Montevideo, Uruguay,
{\tt eleonora@fing.edu.uy}} \ and Yun Zhao\footnote{Department of mathematics, Soochow University, Suzhou 215006, Jiangsu, P.R.China,
 {\tt zhaoyun@suda.edu.cn}}
\end{center}



\begin{abstract}

For a sequence of subadditive potentials,
 a method of choosing
 state points with negative growth rates for an ergodic
dynamical system was given in \cite{dai}. This paper first
generalizes this result to the non-ergodic dynamics, and then
proves that under some mild additional  hypothesis, one can choose
points with negative growth rates from a positive Lebesgue measure
set, even if the system does not preserve any measure that is
absolutely continuous with respect to Lebesgue measure.
\end{abstract}



\section{Introduction}

Let $f:M\rightarrow M$ be a continuous map on a compact,
finite-dimensional manifold $M$, and $m$ a normalized Lebesgue
measure on $M$.  We denote with $\mathcal{M}$ the set of all the
Borel probability measures on $M$, provided with the weak$^*$
topology, and with $\mathrm{dist}^*$ a metric inducing this
topology. The terms $\mathcal{M}_f$ and $\mathcal{E}_f$ denote the
space of $f-$invariant Borel probability measures and the set of
$f-$invariant ergodic Borel probability measures, respectively.

For each point $x\in M$, we define the empirical   measures
\[
\delta_{x,n}=\frac 1 n \sum_{j=0}^{n-1}\delta_{f^j(x)}
\]
where $\delta_x$ is the Dirac measure at $x$. We denote with
$\mathcal{V}_f(x)$
  the set of all the Borel probabilities in $\mathcal{M}$
that are the weak$^*$ limits of the empirical measures. It is well
known  that $\mathcal{V}_f(x) \subset \mathcal{M}_f$.

A sequence $\Phi=\{ \phi_n\}_{n\geq 1}$ of continuous real
functions  is a  subadditive potential on $M$, if
\[
\phi_{n+m}(x)\leq \phi_n(x) +\phi_m(f^n x)
~~~~{\mathrm{for}}~{\mathrm{all}}~x\in M,~ n,m\in {\mathbb{N}}.
\]
For $\mu\in \mathcal{M}_f$, it follows, from Kingman's
sub-additive ergodic theorem (see \cite{kin} or \cite[theorem
10.1]{wal}), that
\[ \Phi^*(x):=\lim\limits_{n\rightarrow\infty}
 \frac{1}{n}  \phi_n (x)~~\hbox{for}~\mu-{\hbox{a.e.}}~x\in M
 \]
and $\int \Phi^*(x) \mathrm{d}\mu=\inf_{n\geq 1}\frac 1 n \int
\phi_n  \mathrm{d}\mu$. The term $\Phi^*(x)$ is called \em the
growth rate \em of the subadditive potentials $\Phi=\{\phi_n\}$ at
$x$, defined as the existing limit for a set of full measure for
any invariant measure. All along the paper we will assume that the
growth rate  is negative, i.e. $\Phi^*(x)<0$ for $\mu-$almost all
$x$ for one or more invariant measures $\mu$.  We are interested
to select other state points $x \in M$ for which the {\it largest
growth rate} is still negative, namely:
$$\widetilde{\Phi}^*(x):=\limsup_{n\rightarrow\infty}\frac 1 n
\phi_n(x)<0.$$  The points $x \in M$ such that
$\widetilde{\Phi}^*(x) < 0$, if exist,  are called \em optimal
state points \em for the given sequence $\Phi$ of subadditive
potentials. We will say that the set of optimal state points is
\em observable, \em if its Lebesgue measure is positive. We notice
that we are neither  assuming that the system preserves the
Lebesgue measure, nor any measure that is absolutely continuous
with respect to Lebesgue measure.

The problem of the abundance of optimal state points   arises for
example  when studying the stability of linear control systems
\cite{dhx}. We are also interested to describe, and to find
Lebesgue positive subsets of state points   $x \in M$, that are
not necessarily optimal, but for which the  \em smallest growth
rate  \em is negative. Namely:
$$\widehat {\Phi}^*(x) := \liminf_{n \rightarrow \infty} \frac 1 n \phi_n(x) < 0.$$

Dai \cite{dai} gave a method to choose optimal   points. We
rewrite his result in our setting as the following theorem:

\begin{theorem} \em{(\cite{dai}) }\em
Let $f:M\rightarrow M$ be a continuous map on a compact,
finite-dimensional manifold $M$, $\mu$ an $f-$invariant ergodic
measure, and $\Phi=\{ \phi_n\}_{n\geq 1}$  a  subadditive
potential. If the growth rates of $\Phi$ satisfy
\[
\Phi^*(x):=\lim\limits_{n\rightarrow\infty}
 \frac{1}{n}  \phi_n (x)<0~~~\mu-{\hbox{a.e.}}~x\in M
\]
then $\widetilde{\Phi}^*(x)<0$ for all $x$ in the basin $B(\mu)$
of the measure $\mu$, where   $B(\mu):=\{x\in M:
\mathcal{V}_f(x)=\{\mu\}\}.$
\end{theorem}

The above theorem states that all the points in the basin of an
ergodic measure $\mu$ have negative largest growth rates. A
natural question arises with respect to all the other invariant
measures:

 {\bf Question 1:} {\it How can we ensure the existence of state
points $x$ that have   negative largest growth rates,  if the
measure is $f-$invariant but not necessarily ergodic? }

For an experimenter  it would be interesting to know whether the
subadditive potentials have negative largest (or at least
smallest) growth rates in a set with positive Lebesgue measure.
 To achieve that result and after   Dai's Theorem it would be enough that
 the basin $B(\mu)$ of the ergodic measure $\mu$ has positive Lebesgue measure.
 But, although $B(\mu)$ has full $\mu-$measure,
its Lebesgue measure may be zero, unless $\mu$ is   {\it physical
 or SRB}. Nevertheless, general continuous dynamical systems  may
not have such a physical or SRB measure.  Thus, it arises the
following
 nontrivial question:

{\bf Question 2:} {\it Under what conditions, even if no physical
or SRB measure exists,  the subadditive potentials have negative
largest   growth rates (or at least negative smallest growth
rates) on a positive Lebesgue measure set?}

The following theorem provides a strong result in this direction.
It was proved independently by Schreiber \cite[theorem 1]{sch} and
Sturman and Stark \cite[theorem 1.7]{ss}, and more recently Dai
\cite{dai} gave a   proof  by a simple method. We rewrite it in
our setting, with the following statement:

\begin{theorem} \em{ (\cite{sch,ss})}\em \label{theoremShcreiber} Let $f:M\rightarrow M$ be a continuous map on a compact,
finite-dimensional manifold $M$, and $\Phi=\{ \phi_n\}_{n\geq 1}$
a  subadditive potential. If for each $f-$invariant ergodic
measure $\mu$  the growth rates
\[
\Phi^*(x):=\lim\limits_{n\rightarrow\infty}
 \frac{1}{n}  \phi_n (x)<0~~~\mu-{\hbox{a.e.}}~x\in M
\]
then $\widetilde{\Phi}^*(x)<0$  for  all $x\in M$.
\end{theorem}
 Cao
\cite{cao} also extended the result in the latter theorem to
random dynamical systems.

As said above, when observing the system  the experimenter may
just need to know that the subadditive potentials have negative
growth rates in a positive Lebesgue measure set of state points
$x$, instead in the whole manifold. That is why we are interested
to find the conditions, weaker than the hypothesis  of Theorem
\ref{theoremShcreiber}, to ensure that $\widetilde \Phi^*(x)< 0$,
or at least $\widehat \Phi^*(x) < 0$, just for a Lebesgue-positive
set of state points $x \in M$.

The questions 1 and 2 above are the   motivations    and tasks of
this paper. To search for an answer to Question 2  we apply the
recent works in \cite{ce} and \cite{c}. Motivated and applying
those results, and adapting the arguments of Dai \cite{dai},
   we    give     positive answers to
Questions 1 and 2  in  Theorems \ref{theoremA}, \ref{theoremB}
and \ref{theoremC},    and Corollaries \ref{corollaryA},
\ref{CorollaryB} and \ref{corollaryC} of this paper.

The exposition is organized as follows: Section 2 provides some
preliminary definitions and the statements of our main results in
Theorems \ref{theoremA}, \ref{theoremB} and \ref{theoremC}.
Section 3 provides their detailed proofs and their corollaries.
Section 4 gives   additional remarks about  the main results, and
some examples to illustrate them.

\section{Preliminaries and statements of the new results}

This section provides  some definitions  and  the statements of
the main results of this paper.

To give a positive answer to question 1, we give the definition of
strong basin of $\epsilon-$attraction of an invariant measure, as
follows:

\begin{definition} \label{strongobs} \em
Given a probability measure $\mu\in \mathcal{M}$  and a real
number $\epsilon>0$, the following set
\[
\mathcal{S}_{\epsilon}(\mu):=\{ x\in M: \mathcal{V}_f(x)\subset
\mathcal{N}_{\epsilon}(\mu)\}
\]
is called \em strong basin of $\epsilon-$attraction of $\mu$, \em
where $\mathcal{N}_{\epsilon}(\mu)$ denotes the
$\epsilon-$neighborhood of $\mu$ under the metric
$\mathrm{dist}^*$. Furthermore, if the set
$\mathcal{S}_{\epsilon}(\mu)$ has positive Lebesgue measure for
all $\epsilon>0$, then the measure $\mu$ is called   \em strong
observable \em  for $f$.
\end{definition}

To answer question 2, we first recall  the definition of
observable measure   introduced   in \cite{ce}.
\begin{definition} \label{observable} \em
Given a probability measure $\mu\in \mathcal{M}$ and a real number
$\epsilon>0$, the following set
\[
A_{\epsilon}(\mu):=\{ x\in M: \mathcal{V}_f(x)\cap
\mathcal{N}_{\epsilon}(\mu)\neq \emptyset\}
\]
is called \em basin of $\epsilon-$attraction of $\mu$. \em
Furthermore, if the set  $A_{\epsilon}(\mu)$ has positive Lebesgue
measure for all $\epsilon>0$, then the measure $\mu$ is called \em
observable \em for $f$. Let $\mathcal{O}_f$ denote the set of all
the observable measures for $f$.
\end{definition}

\begin{remark}

\em From the definitions above, it is immediate that   any strong
observable measure is observable. Nevertheless  given an
$f-$invariant measure $\mu$  its (strong)  basin of
$\epsilon-$attraction may be empty for all $\epsilon>0$. If $\mu$
is ergodic  then  its
 strong  basin of $0-$attraction (which is obviously   included in
its strong  basin of $\epsilon-$attraction for all positive
$\epsilon$) is not empty since it includes $\mu-$almost all the
points.
\end{remark}

It is easy to check that each (strong) observable measure is
$f-$invariant, since it can be approximated by invariant measures
and $\mathcal{M}_f$ is weak$^*$ compact. See \cite{ce} for more
properties of observable measures. In order to give a class of
systems for which the  strong  basins of $\epsilon-$attraction are
not empty for all $\epsilon>0$ and for all $\mu \in
\mathcal{M}_f$, we first recall Bowen's specification property
(see definition 18.3.8 in \cite{kh}):

\begin{definition} \em
 A continuous map $f:M\rightarrow M$ satisfies \em
Bowen's specification property \em if for each $\epsilon>0$, there
exists an integer $m=m(\epsilon)$ such that for any finite
collection $\{I_j:=[a_j,b_j]\subset \mathbb N:j=1,2,...,k\}$ of
finite intervals of natural numbers such that  $a_{j+1}-b_j\geq
m(\epsilon)$ for $j=1,2,...,k-1$, for any $x_1,x_2,...,x_k$ in
$M$, and for any $p\geq b_k-a_1+m(\epsilon)$, there exists a
periodic point $x\in M$ of
  period at least $p$, satisfying
\begin{eqnarray*}
d(f^{l+a_j}(x),f^l(x_j))<\epsilon  \mbox{ for all
$l=0,1,...,b_j-a_j $ and every $j=1,2,...,k.$}
\end{eqnarray*}
\end{definition}

\begin{proposition}Let $f:M\rightarrow M$ be a continuous map on a compact
finite-dimensional manifold $M$. Assume that $f$ satisfies Bowen's
specification property. Then for all $f-$invariant measure $\mu$
and for all $\epsilon >0$,
 $$\mathcal{S}_{\epsilon}(\mu)\neq \emptyset \ \mbox{ and } \  A_{\epsilon}(\mu)\neq \emptyset.$$ \label{proposition1}
\end{proposition}
\begin{proof} Since
$\mathcal{S}_{\epsilon}(\mu) \subset
 A_{\epsilon}(\mu)$, it is enough to prove that $\mathcal{S}_{\epsilon}(\mu)\neq
\emptyset$ for all $\epsilon>0$.  By hypothesis $f$ satisfies
Bowen's specification property. Then, the set of ergodic measures
$\mathcal{E}_f$ (precisely the subset of the invariant measures
that are supported on  periodic orbits)
 is dense in the space of $f-$invariant measures
$\mathcal{M}_f$ (see the main theorem in \cite{sig1} or
\cite{sig2} for a proof). Thus, for  any $\mu\in \mathcal{M}_f$
and any $\epsilon>0$, there exists an ergodic measure $\nu\in
\mathcal{N}_{\epsilon}(\mu)$. The basin $B(\nu)$  is not empty
since $\nu$ is ergodic and thus
 $\mathcal{V}_f(x) = \{\nu\}$ for $\nu$-a.e. point $x \in M$. Besides, from $\nu \in \mathcal N_{\epsilon}(\mu)$ we obtain that $B(\nu) \subset \mathcal S_{\epsilon}(\mu)$. We deduce that
 $ \mathcal{S}_{\epsilon}(\mu) \neq \emptyset $,  ending the proof.
\end{proof}

In order to answer question 2, we also need to revisit the
definition of observable measures for  a subsystem.

\begin{definition} \em
Let $B\subset M$ be a forward invariant set, i.e. $f(B)\subset B$,
that has positive Lebesgue measure. A probability measure $\mu$ is
\em observable for $f\mid_B$, \em if for all $\epsilon>0$ the
following set
\[
A_{\epsilon}(B, \mu):=\{ x\in B: \mathcal{V}_f(x)\cap
\mathcal{N}_{\epsilon}(\mu)\neq \emptyset\}
\]
has positive Lebesgue measure. Let $\mathcal{O}_{f\mid_B}$ denote
the set of all the observable measures for $f\mid_B$.
\end{definition}

The following definition of Milnor-like attractor was introduced
in \cite{c}.

\begin{definition} \em
Let $K\subset M$ be a nonempty, compact and $f-$invariant set,
i.e. $f^{-1}(K)=K$. We say that $K$ is a \em Milnor-like attractor
\em if the   following set
\[
B(K):=\{x\in M: \liminf_{n\rightarrow\infty}\frac 1 n \sharp
\{0\leq j \leq n-1: f^j(x)\in \mathcal{N}_\epsilon (K)
\}=1,\forall \epsilon>0\}
\]
has positive Lebesgue measure, where $\sharp A$ denotes the
cardinality of a set $A$ and $\mathcal{N}_\epsilon (K)$ denotes
the $\epsilon-$neighborhood of $K$, i.e., $\mathcal{N}_\epsilon
(K):=\bigcup_{x\in K}B(x,\epsilon)$ and $B(x,\epsilon)$ is a ball
of radius $\epsilon$ centered at $x$. The set $B(K)$ is called the
basin of $K$.
\end{definition}

Note that the $\liminf$ equal to one in the definition above,
implies that the limit exists and is equal to one.

For each $0<\alpha\leq 1$, the Milnor-like attractor $K$ is called
\em $\alpha-$observable \em if $m(B(K))\geq \alpha$. An
$\alpha-$observable Milnor-like attractor \em is minimal, \em if
it has no proper subsets that are also $\alpha-$observable
Milnor-like attractor for the same value of $\alpha$.

We restate here, just for completeness, the following result (see
\cite{ce,c} for a proof):

\begin{theorem}\label{srb} \em{ (
\cite{ce,c})}\em Let $f:M\rightarrow M$ be a continuous map on a
compact, finite-dimensional manifold $M$, and let $B\subset M$ be
a forward invariant set that has positive Lebesgue measure. Then
the following properties hold:
\begin{itemize}
\item[(1)] The set $\mathcal{O}_f$ of all observable measures for
$f$ is nonempty, and minimally weak$^*$ compact containing for
Lebesgue almost all $x\in M$, all the weak$^*$ limits of the
convergent subsequences of  empirical   measures.

\item[(2)] For each $0<\alpha\leq 1$, there exists a minimal
$\alpha-$observable Milnor-like attractor.

\item[(3)]The set $\mathcal{O}_{f\mid_B}$ is weak$^*$ compact and
nonempty.

\item[(4)] The set $\mathcal{O}_{f\mid_B}$ is the minimal weak$^*$
compact set in the space $\mathcal{M} $ such that
$\mathcal{V}_f(x)\subset \mathcal{O}_{f\mid_B}$ for Lebesgue
almost all $x\in B$.
\end{itemize}
\end{theorem}

Now we state the main theorems of this paper. We will give their
proofs  in the next section.

The first theorem states that a  subadditive potential has
negative largest  growth rates at all the state points $x$
belonging to the strong basin of $\epsilon-$attraction of an
invariant measure, for some $\epsilon>0$ (and thus for all
$\epsilon >0$ small enough).

\begin{theorem} \label{theoremA} Let $f:M\rightarrow M$ be a continuous map on a
compact, finite-dimensional manifold $M$, $\mu$ an $f-$invariant
measure, and $\Phi=\{ \phi_n\}_{n\geq 1}$ a  subadditive
potential. If the growth rates of $\Phi$ satisfy
\[
\Phi^*(x):=\lim\limits_{n\rightarrow\infty}
 \frac{1}{n}  \phi_n (x)<0~~~\mu-{\hbox{a.e.}}~x\in M,
\]
then there exists $\epsilon>0$ such that $\widetilde{\Phi}^*(x)<0$
for each point $x$ in $\mathcal{S}_{\epsilon}(\mu)$.
\end{theorem}

\begin{remark}
\label{remarkA} \em Combined   with Proposition \ref{proposition1}
(which asserts that ${\mathcal S}_{\epsilon}(\mu) \neq \emptyset$)
the theorem above answers positively Question 1 of the
introduction, for any system that satisfies Bowen's specification
property and for any invariant measure $\mu$.
\end{remark}

The second theorem states that the subadditive potentials   have
negative {\it smallest  growth rates} in the basin of
$\epsilon-$attraction of any invariant measure, for some
$\epsilon>0$.
\begin{theorem}  \label{theoremB} Let $f:M\rightarrow M$ be a continuous map on a compact,
finite-dimensional manifold $M$, $\mu$ an $f-$invariant measure,
and $\Phi=\{ \phi_n\}_{n\geq 1}$ a subadditive potential. If the
growth rates of $\Phi$ satisfies
\[
\Phi^*(x):=\lim\limits_{n\rightarrow\infty}
 \frac{1}{n}  \phi_n (x)<0~~~\mu-{\hbox{a.e.}}~x\in M
\]
then there exists $\epsilon>0$ such that $\widehat{\Phi}^*(x)<0$
for each point $x$ in $A_{\epsilon}(\mu)$.
\end{theorem}

\begin{remark}
\label{remarkB} \em

From theorem \ref{srb} we deduce that there always exist
observable invariant measures $\mu$. Applying  theorem
\ref{theoremB} to those
 measures, we obtain that the {\it smallest growth rates}
of the subadditive potential is negative on a set of positive
Lebesgue measure, for any continuous system.
\end{remark}

Moreover, the following theorem states that under mild stronger
conditions  the {\it largest
 growth rates}   are also negative for Lebesgue almost
all the points in the basin of any Milnor-like attractor:

\begin{theorem} \label{theoremC} Let $f:M\rightarrow M$ be a continuous map on a
compact, finite-dimensional manifold $M$, and $\Phi=\{
\phi_n\}_{n\geq 1}$ a  subadditive potential. Assume that $K$ is
an $\alpha-$observable Milnor-like attractor for some
$0<\alpha\leq 1$, and $B(K)$ is its  basin. If to each observable
measure $\mu\in \mathcal{O}_{f\mid_{B(K)}}$, the growth rates
\[
\Phi^*(x):=\lim\limits_{n\rightarrow\infty}
 \frac{1}{n}  \phi_n (x)<0~~~\mu-{\hbox{a.e.}}~x\in M
\]
then $\widetilde{\Phi}^*(x)<0$ for $m-$almost every $x\in B(K)$.
\end{theorem}

\begin{remark}
\label{remarkC} \em

Note that the Lebesgue measure of the basin $B(K)$ of the
Milnor-like attractor $K$ in the latter theorem is larger than or
equal to $\alpha$. Thus the {\it largest growth rates} of the
subadditive potentials is negative on a set with  Lebesgue measure
that is at least
 equal to $\alpha$.  This implies that the optimal state points $x$ (namely the points
 for which the largest growth rates are negative) cover a set that is Lebesgue $\alpha$-observable
 in the  manifold $M$. This answers positively Question 2 of the Introduction.

 Moreover, if  $K$ is a $1-$observable Milnor-like
attractor (such $K$ always exists after part (2) of theorem
\ref{srb}), then Theorem \ref{theoremC} asserts that the {\it
largest growth rates} of the subadditive potentials are negative
Lebesgue almost everywhere.
\end{remark}

To end this section, we state a useful known lemma. It appears in
many places; see  for example \cite{cao}. We give a proof here
just for completeness.

\begin{lemma} \label{sub}
Let $f:M\rightarrow M$ be a continuous map on a compact,
finite-dimensional manifold $M$, and $\Phi=\{ \phi_n\}_{n\geq 1}$
a  subadditive potential. Fix any positive integer $l$. Then
\[
\phi_n(x)\leq C+ \sum_{i=0}^{n-1} \frac{1}{l}\phi_l(f^ix)~~\forall
x\in M
\]
where $C$ is a constant depending only on $l$.
\end{lemma}
\begin{proof} Fix a positive integer $l$. For each natural number $n $, we
 write  $n=sl+k$, where $0\leq s, 0\leq k<l$. Then, for any
integer  $0\leq j<l$  we have
\[
\phi_n(x)\leq \phi_j(x)+\phi_l(f^jx)+\cdots +
\phi_l(f^{(s-2)l}f^jx)+\phi_{k+l-j}(f^{(s-1)l}f^jx),
\]
where $\phi_0(x)\equiv 0$. Let $C_1=\max_{j=1,\cdots
2l}||\phi_j||_{\infty}$.  Adding $\phi_n(x)$ when $j$ takes all
the natural values from $0$ to $l-1$, we have
\[
l\phi_n(x)\leq 2lC_1 + \sum_{i=0}^{(s-1)l-1} \phi_l(f^ix).
\]
Hence
\begin{eqnarray*}
\phi_n(x)\leq 2C_1+ \sum_{i=0}^{(s-1)l-1}
\frac{1}{l}\phi_l(f^ix)\leq 4C_1+ \sum_{i=0}^{n-1}
\frac{1}{l}\phi_l(f^ix).
\end{eqnarray*}
Choosing $C=4C_1$ the desired result follows.
\end{proof}

\section{Proofs of the main results}
This section provides the proofs of the theorems in section 2.

\subsection{Proof of theorem \ref{theoremA}}
\begin{proof}
Let $\mu$ be an $f-$invariant measure   that satisfies the
hypothesis of theorem \ref{theoremA}. The arguments here are
similar to those  of Dai in \cite{dai}. Let $\psi_n(x)=\max \{-n,
\phi_n(x)\}$ for all $n\geq 1$ and each $x\in M$. It is easy to
see that the sequence of functions  $\Psi=\{\psi_n\}$ is
subadditive. Set
\[
\widetilde{\Psi}^*(x):=\limsup_{n\rightarrow\infty}\frac 1 n
\psi_n(x) ~~~~\forall x \in M.
\]
Under the hypothesis of theorem \ref{theoremA}, it follows that
$\widetilde{\Psi}^*(x)<0$ for $\mu-$almost every $x\in M$. Since
$\psi_n(x)\geq \phi_n(x)$ for all $n\geq 1$ and all $x\in M$, to
prove theorem \ref{theoremA} it is enough to show that
$\widetilde{\Psi}^*(x)<0$ for all $x \in
\mathcal{S}_{\epsilon}(\mu)$ for some $\epsilon>0$.

Using the definition of $\Psi$ and the subadditivity of $\Psi$, we
have
\[
-1\leq \frac 1 n \psi_n(x)\leq ||\psi_1||_\infty~~~\forall x\in M.
\]
It follows from the Fatou lemma
 that
\[
\inf_{n\geq 1}\frac 1 n \int \psi_n
\mathrm{d}\mu=\lim_{n\rightarrow\infty}\frac 1 n \int \psi_n
\mathrm{d}\mu\leq \int \limsup_{n\rightarrow\infty}\frac 1 n
\psi_n(x) \mathrm{d}\mu=\int \widetilde{\Psi}^*(x)
\mathrm{d}\mu<0.
\]
 Therefore, there exists an integer $l\geq 1$
such that
\[
-1\leq \frac 1 l \int \psi_l \mathrm{d}\mu<0.
\]
For some sufficiently small $\eta >0$, say $\eta <\frac{|\frac 1 l
\int \psi_l \mathrm{d}\mu|}{2}$, fix a positive number
$\epsilon>0$ such that
\[
\mathrm{dist}^{*}(\mu,\nu)\leq \epsilon \Rightarrow | \int \frac 1
l \psi_l \mathrm{d}\mu-\int \frac 1 l \psi_l \mathrm{d}\nu|<\eta.
\]
If the strong basin of $\epsilon$-attraction of $\mu$ is empty,
i.e., $\mathcal{S}_{\epsilon}(\mu)=\emptyset$, then there is
nothing to prove. Otherwise, let
\[
\mathcal{D}=\{x\in \mathcal{S}_{\epsilon}(\mu):
\limsup_{n\rightarrow\infty}\frac 1 n \psi_n(x)\geq 0\}.
\]
We will prove that $\mathcal{D}=\emptyset$. Assume by
contradiction that there exists $x_0\in \mathcal{D}$. Since
$x_0\in \mathcal{D}\subset \mathcal{S}_{\epsilon}(\mu)$, choose a
subsequence  of integers $\{n_i\}$ such that $\delta_{x_0,n_i}$
converges weakly to a measure $\tilde{\mu}$ and
$\lim_{i\rightarrow\infty}\frac{1}{n_i}\psi_{n_i}(x_0)=\limsup_{n\rightarrow\infty}\frac
1 n \psi_n(x_0)$.  Note that $\tilde{\mu}\in
\mathcal{V}_f(x_0)\subset\mathcal{N}_{\epsilon}(\mu)$, i.e.,
$\mathrm{dist}^*(\mu,\tilde{\mu} )\leq \epsilon$. . It follows
that
\begin{eqnarray*}
0>\int \frac 1 l \psi_l \mathrm{d}\mu+\eta >\int \frac 1 l \psi_l
\mathrm{d}\tilde{\mu}=\lim_{i\rightarrow\infty} \int \frac 1 l
\psi_l \mathrm{d} \delta_{x_0,n_i}=\lim_{i\rightarrow\infty}
\frac{1}{n_i} \sum_{j=0}^{n_i-1} \frac 1 l \psi_l(f^jx_0).
\end{eqnarray*}
Using lemma \ref{sub}, we have
\[
\lim_{i\rightarrow\infty}\frac 1 n_i \psi_{n_i}(x_0)\leq
\lim_{i\rightarrow\infty}\frac 1 n_i \sum_{j=0}^{n_i-1} \frac 1 l
\psi_l(f^jx_0).
\]
Note that $x_0\in \mathcal{D}$, we have
\[
0>\lim_{i\rightarrow\infty} \frac 1 n_i \sum_{j=0}^{n_i-1} \frac 1
l \psi_l(f^jx_0)\geq \lim_{i\rightarrow\infty}\frac 1 n_i
\psi_{n_i}(x_0)\geq 0
\]
which is a contradiction. This completes the proof of theorem
\ref{theoremA}.
\end{proof}

\begin{corollary} \label{corollaryA}
Let $f:M\rightarrow M$ be a continuous map on a compact,
finite-dimensional manifold $M$,  and  $\Phi=\{ \phi_n\}_{n\geq
1}$  a  subadditive potential. Assume that there exists   a strong
observable measure $\mu$. If the growth rates of $\Phi$ satisfies
\[
\Phi^*(x):=\lim\limits_{n\rightarrow\infty}
 \frac{1}{n}  \phi_n (x)<0~~~\mu-{\hbox{a.e.}}~x\in M
\]
then   $\widetilde{\Phi}^*(x)<0$  on a set with positive Lebesgue
measure.
\end{corollary}
\begin{proof}First note that $\mu$ is also $f-$invariant, thus
there exists $\epsilon>0$ such that $\widetilde{\Phi}^*(x)<0$ on
the set $\mathcal{S}_\epsilon(\mu)$, i.e.,  the strong basin of
$\epsilon-$attraction of $\mu$. And since $\mu$ is strong
observable, we have $m(\mathcal{S}_\epsilon(\mu))>0$. This
completes the proof of the corollary.
\end{proof}

\subsection{Proof of theorem \ref{theoremB}}
\begin{proof} Let $\mu$ be an $f-$invariant measure   that satisfies the hypothesis
of theorem \ref{theoremB}.  Let $\psi_n(x)=\max \{-n, \phi_n(x)\}$
for all $n\geq 1$ and each $x\in M$. As in the previous proof, the
sequence
  $\Psi=\{\psi_n\}$ is subadditive. Set
\[
\widehat{\Psi}^*(x):=\liminf_{n\rightarrow\infty}\frac 1 n
\psi_n(x) ~~~~\forall x \in M.
\]
Under the hypothesis of theorem \ref{theoremB}, it follows that
$\widehat{\Psi}^*(x)  <0$ for $\mu-$almost all $x$. Since
$\psi_n(x)\geq \phi_n(x)$ for all $n\geq 1$ and all $x\in M$, to
prove theorem \ref{theoremB} it is enough to show that
$\widehat{\Psi}^*(x)<0$ for each point in $A_{\epsilon}(\mu)$ for
some $\epsilon>0$.

Using the definition of $\Psi$ and the subadditivity of $\Psi$, we
have
\[
-1\leq \frac 1 n \psi_n(x)\leq ||\psi_1||_\infty~~~\forall x\in M.
\]
It follows from the Fatou lemma
 that
\[
\inf_{n\geq 1}\frac 1 n \int \psi_n
\mathrm{d}\mu=\lim_{n\rightarrow\infty}\frac 1 n \int \psi_n
\mathrm{d}\mu\leq \int \limsup_{n\rightarrow\infty}\frac 1 n
\psi_n(x) \mathrm{d}\mu<0.
\]
The last inequality holds since
$\limsup\limits_{n\rightarrow\infty}\frac 1 n \psi_n(x)  <0$ for
$\mu-$almost all $x\in M$. Therefore, there exists an integer
$l\geq 1$ such that
\[
-1\leq \frac 1 l \int \psi_l \mathrm{d}\mu<0.
\]
For some sufficiently small $\eta >0$, say $\eta <\frac{|\frac 1 l
\int \psi_l \mathrm{d}\mu|}{2}$, fix a positive number
$\epsilon>0$ such that
\[
\mathrm{dist}^{*}(\mu,\nu)\leq \epsilon \Rightarrow | \int \frac 1
l \psi_l \mathrm{d}\mu-\int \frac 1 l \psi_l \mathrm{d}\nu|<\eta.
\]
If the basin of $\epsilon$-attraction of $\mu$ is empty, i.e.,
$A_{\epsilon}(\mu)=\emptyset$, then there is nothing to prove.
Otherwise, let
\[
\mathcal{D}=\{x\in A_{\epsilon}(\mu):
\liminf_{n\rightarrow\infty}\frac 1 n \psi_n(x)\geq 0\}.
\]
We will prove that $\mathcal{D}=\emptyset$. Assume by
contradiction that there exists $x_0\in \mathcal{D}$. Since
$x_0\in \mathcal{D}\subset A_{\epsilon}(\mu)$, there exists
$\tilde{\mu}\in \mathcal{V}_f(x_0)$ such that
$\mathrm{dist}^*(\mu,\tilde{\mu} )\leq \epsilon$. Choose a
subsequence  of integers $\{n_i\}$ such that $\delta_{x_0,n_i}$
converges weakly to the measure $\tilde{\mu}$. It follows that
\begin{eqnarray*}
0>\int \frac 1 l \psi_l \mathrm{d}\mu+\eta >\int \frac 1 l \psi_l
\mathrm{d}\tilde{\mu}=\lim_{i\rightarrow\infty} \int \frac 1 l
\psi_l \mathrm{d} \delta_{x_0,n_i}=\lim_{i\rightarrow\infty} \frac
1 n_i \sum_{j=0}^{n_i-1} \frac 1 l \psi_l(f^jx_0).
\end{eqnarray*}
Using lemma \ref{sub}, we have
\[
\liminf_{i\rightarrow\infty}\frac 1 n_i \psi_{n_i}(x_0)\leq
\lim_{i\rightarrow\infty}\frac 1 n_i \sum_{j=0}^{n_i-1} \frac 1 l
\psi_l(f^jx_0).
\]
Note that $x_0\in \mathcal{D}$, we have
\[
0>\lim_{i\rightarrow\infty} \frac 1 n_i \sum_{j=0}^{n_i-1} \frac 1
l \psi_l(f^jx_0)\geq \liminf_{i\rightarrow\infty}\frac 1 n_i
\psi_{n_i}(x_0)\geq \liminf_{n\rightarrow\infty}\frac 1 n
\psi_{n}(x_0)\geq 0
\]
which is a contradiction. This completes the proof of theorem
\ref{theoremB}.
\end{proof}

\begin{corollary} \label{CorollaryB}
Let $f:M\rightarrow M$ be a continuous map on a compact,
finite-dimensional manifold $M$,  and $\Phi=\{ \phi_n\}_{n\geq 1}$
a  subadditive potential. Let $\mu$ be any (always existing)
observable measure for $f$. If the growth rates of $\Phi$ satisfy
\[
\Phi^*(x):=\lim\limits_{n\rightarrow\infty}
 \frac{1}{n}  \phi_n (x)<0~~~\mu-{\hbox{a.e.}}~x\in M
\]
then   $\widehat{\Phi}^*(x)<0$  on a set with positive Lebesgue
measure.
\end{corollary}
\begin{proof}First note that $\mu$ is also $f-$invariant, thus
there exists $\epsilon>0$ such that $\widehat{\Phi}^*(x)<0$ on the
set $A_\epsilon(\mu)$, i.e.,  the basin of $\epsilon-$attraction
of $\mu$. And since $\mu$ is observable, we have
$m(A_\epsilon(\mu))>0$. This completes the proof of the corollary.
\end{proof}

\subsection{Proof of theorem \ref{theoremC}}
\begin{proof} As in the proof of theorem \ref{theoremB},
define  $\psi_n(x)=\max \{-n, \phi_n(x)\}$ for all $n\geq 1$ and
each $x\in M$. Then the sequence of functions $\Psi=\{\psi_n\}$ is
a family of continuous functions which is  subadditive. Set
\[
\widetilde{\Psi}^*(x):=\limsup_{n\rightarrow\infty}\frac 1 n
\psi_n(x) ~~~~\forall x \in M.
\]
Under the hypothesis of theorem \ref{theoremC}, to  each
observable measure $\mu\in \mathcal{O}_{f\mid_B}$, it holds that
$\widetilde{\Psi}^*(x)<0$ for $\mu-$almost every $x\in M$. By the
definition of $\Psi$, it is easy to see that
$\widetilde{\Psi}^*(x)\geq \widetilde{\Phi}^*(x)$ for each $x\in
M$. So,  to prove theorem \ref{theoremC} it is enough to show that
$\widetilde{\Psi}^*(x)<0$ for $m-$almost every $x\in B(K)$, where
$m$ denotes the Lebesgue measure.

Since   the Milnor-attractor $K$ is $\alpha-$observable, its basin
$B(K)$ satisfies $m(B(K))\geq \alpha$. It is easy to check that
the basin $B(K)$ is forward invariant. Thus,  each of the
observable measures in the set ${\mathcal O} \mid _{f\mid_{B(K)}}$
satisfies Theorem \ref{srb}.

Let $\mathcal{D}=\{x\in B(K):
\limsup\limits_{n\rightarrow\infty}\frac 1 n \psi_n(x)\geq 0\}$.
To end the proof of Theorem \ref{theoremC} it  is now enough to
show that $m(\mathcal{D})=0$.

Assume by contradiction that $m(\mathcal{D})>0$. By the fourth
item of theorem \ref{srb}, we can choose a point $x_0\in
\mathcal{D}$ such that $\mathcal{V}_f(x_0)\subset
\mathcal{O}_{f\mid_{ B(K)}}$. We can take a subsequence of
integers $\{n_i\}$ such that $\delta_{x_0,n_i}$ converges weakly
to the measure $\mu\in \mathcal{V}_f(x_0)$ and
$\lim_{i\rightarrow\infty}\frac{1}{n_i}\psi_{n_i}(x_0)=\limsup_{n\rightarrow\infty}\frac
1 n \psi_n(x_0)$. As
\[
-1\leq \frac 1 n \psi_n(x)\leq ||\psi_1||_\infty~~~\forall x\in M,
\]
it follows from the Fatou lemma  that
\[
\inf_{n\geq 1}\frac 1 n \int \psi_n
\mathrm{d}\mu=\lim_{n\rightarrow\infty}\frac 1 n \int \psi_n
\mathrm{d}\mu\leq \int \limsup_{n\rightarrow\infty}\frac 1 n
\psi_n(x) \mathrm{d}\mu<0
\]
since $\mu\in \mathcal{O}_{f\mid_{ B(K)}}$, and by the definition
of $\Psi$  and the hypothesis of theorem \ref{theoremC}, we obtain
that
\[\limsup_{n\rightarrow\infty}\frac 1 n \psi_n(x) <0 \] for $\mu-$almost every point $x\in M$.  So, there exists an integer $l\geq 1$ such that
$ -1\leq \frac 1 l \int \psi_l \mathrm{d}\mu<0$, and further
\[
0>\frac 1 l \int \psi_l
\mathrm{d}\mu=\lim_{i\rightarrow\infty}\int \frac 1 l \psi_l
\mathrm{d}\delta_{x_0,n_i}=\lim_{i\rightarrow\infty} \frac{1}{n_i}
\sum_{j=0}^{n_i-1}\frac 1 l \psi_l(f^jx_0).
\]
Using lemma \ref{sub}, we have
\[
\lim_{i\rightarrow\infty}\frac{1}{n_i} \psi_{n_i}(x_0)\leq
\lim_{i\rightarrow\infty}\frac{1}{n_i} \sum_{j=0}^{n_i-1} \frac 1
l \psi_l(f^jx_0).
\]
Note that $x_0\in \mathcal{D}$, we have
\[
0>\lim_{i\rightarrow\infty} \frac{1}{n_i} \sum_{j=0}^{n_i-1} \frac
1 l \psi_l(f^jx_0)\geq \lim_{i\rightarrow\infty}\frac{1}{n_i}
\psi_{n_i}(x_0)\geq 0
\]
which is a contradiction. This completes the proof of theorem
\ref{theoremC}.
\end{proof}

Using the first item of theorem \ref{srb} and the same arguments
as in the proof of theorem \ref{theoremC}, we have the following
corollary.
\begin{corollary} \label{corollaryC} Let $f:M\rightarrow M$ be a continuous map on a
compact, finite-dimensional manifold $M$, and  $\Phi=\{
\phi_n\}_{n\geq 1}$
 a subadditive potential.  If for all observable
measure $\mu\in \mathcal{O}_{f}$, the growth rates
\[
\Phi^*(x):=\lim\limits_{n\rightarrow\infty}
 \frac{1}{n}  \phi_n (x)<0~~~\mu-{\hbox{a.e.}}~x\in M
\]
then $\widetilde{\Phi}^*(x)<0$ for Lebesgue almost all $x\in M$.
\end{corollary}

\section{Examples and additional remarks}
In \cite{ce} it is proved that observable measures exist for all
continuous systems. Nevertheless, the following  example
(attributed to Bowen \cite{gk,takens} and early cited in
\cite{ta}) shows that not all continuous dynamical systems (indeed
not all $C^2$ systems) have strong observable measures. So, in
this example   Corollary \ref{corollaryA} can not be applied.
Nevertheless we will prove that it still satisfies the final
assertion of that Corollary, since for Lebesgue almost all $x \in
M$, the growth rate $\widetilde \Phi^*(x) < 0$.

\begin{example} \em \label{exampleBowen}
Consider a $C^2$ diffeomorphism  $f$ in a compact ball $M$ of
$\mathbb{R}^2$ with two hyperbolic saddle points $A$ and $B$ in
the boundary $\partial M$ of $M$ such that (half) the unstable
global manifold $W^u(A)\setminus \{A\}$ is an embedded $C^2$ arc
that coincides with (half) the stable global manifold
$W^s(B)\setminus \{B\}$, conversely $W^s(A)\setminus
\{A\}=W^u(B)\setminus \{B\}$, and besides $\partial M = W^u(A)
\cup W^u(B)$. Take $f$ such that there exists a source $C\in U$
where $U$ is the topological open ball with boundary $W^u(A)\cup
W^u(B)$. One can choose $f$ such that for all $x\in U$ the
$\alpha-$limit is $\{C\}$ and the $\omega-$limit contains
$\{A,B\}$. See figure 1 in \cite{takens}. If the eigenvalues of
the derivative of $f$ at $A$ and $B$ are adequately chosen as
specified in \cite{gk,takens}, then  the sequence of empirical
measures  for any $x\in U\setminus \{C\}$ is not convergent. It
has at least two subsequences convergent to different convex
combinations of the Dirac measures $\delta_A$ and $\delta_B$. The
systems, as proved in \cite{gk}, satisfies the following property:

There exists a segment $\Gamma$ in the space of $f-$invariant
 measures, such that $\Gamma$ is a family of convex combinations of the two Dirac
 measures $\delta_A$ and $\delta_B$,  and $\mathcal{V}_f(x)=\Gamma$ for Lebesgue almost all points
 $x$.

 Therefore, as a corollary of the result above, we obtain:
\begin{proposition}
In the example \em \ref{exampleBowen}: \em
\begin{itemize}

\item[(A)] For Lebesgue almost all points, the sequence of
empirical measures does not converge.
 \item[(B)] All measures in
$\Gamma$ are observable according to definition \ref{observable}
and are the only observable measures,
 \item[(C)] There does not
exist strong observable measures according to definition
\ref{strongobs}. \em
\end{itemize}
\end{proposition}
\begin{proof}
(A) is immediate from the fact  that $\mathcal{V}_f(x)=\Gamma$ for
Lebesgue almost all points
 $x$. We refer the proof of (B) to the example 5.5 in
  \cite{ce}. Finally, let us prove (C). Since for all $x\in U\setminus
 \{C\}$  the sequence of empirical   measures has at least two subsequences convergent to different convex
combinations of the Dirac measures $\delta_A$ and $\delta_B$, no
invariant  measure satisfies Definition \ref{strongobs} of strong
observability. In other words, there does not exist   strong
observable measures, because for any invariant measure the strong
basin of $\epsilon-$attraction is empty.\end{proof}
\begin{remark} \em
The proposition above shows that there exist dynamical systems for
which Theorem \ref{theoremA} and Corollary \ref{corollaryA} do not
give   information about the existence of optimal  state points
for the subadditive potentials. Nevertheless, if $\Phi^*(x) <0$
for $\mu$-a.e. just for one (not necessarily ergodic)  invariant
measure $\mu$, and if $\mu$ is some of the always existing
observable measures, then for those systems Theorem \ref{theoremB}
and Corollary \ref{CorollaryB} still ensure the existence of a
Lebesgue positive set of state points with negative smallest
growth rates.
\end{remark}
Moreover, in Example \ref{exampleBowen} we still have the
following very strong statement:

\begin{proposition} Let $f$ be the Bowen's example defined in \em
Example \ref{exampleBowen}, \em  $\mu$  an observable measure  for
$f$, and  $\Phi$ a subadditive potential with the growth rate
$\Phi^*(x) <0$ for $\mu$-a.e. $x$. Then, the largest growth rate
$\widetilde{\Phi}^*(x)$ is negative for Lebesgue almost all $x \in
M$. \em\end{proposition}

\begin{proof}
Any observable measure in this example, i.e. any $\mu \in \Gamma$,
has exactly two ergodic components, that are $\delta_A$ and
$\delta_B$, and has basin of $\epsilon-$attraction
$A_{\epsilon}(\mu)$ that covers Lebesgue almost all $M$.
 Since $\Phi^*(x) < 0$ for $\mu$-a.e. $x$, $\Phi^*(A) < 0$ and $\Phi^*(B) < 0$, because   $\mu$ is supported on $\{A,B\}$.
Therefore $\Phi^*(x) < 0$    $\nu$-a.e. for all other observable
measure $\nu$, because $\nu$ is a convex combination of $\delta_A$
and $\delta _B$. After Corollary \ref{corollaryC}, the largest
growth rate $\widetilde{\Phi}^*(x)$ is negative for Lebesgue
almost all $x \in M$.
\end{proof}

 \begin{remark} \em

 The proof above shows how Theorem \ref{theoremC} and its Corollary \ref{corollaryC} are   powerful  results,
  particularly useful if neither physical nor strong observable measure  exists. In fact, even if   the set of the observable measures components is
   uncountable (as in Example \ref{exampleBowen}),  the set of all their ergodic components may be finite and still the conclusion that
    Lebesgue almost all points are optimal states for a given subadditive potential, may hold. We recall from Definition \ref{strongobs} that if no
     strong observable measure exists, then no physical measure exists. And if no physical measure exists, then the (never empty) set of observable measures
      is necessarily uncountable (for a proof see \cite{ce}).
\end{remark}
\end{example}

\begin{example} \label{exampleMis} \em
In theorem 3.4   of \cite{mi} Misiurewicz proved that there exists
a $C^0$ topologically expanding map $f$ in the circle $S^1$ such
that for Lebesgue almost all $x \in S^1$ the limit set
  $\mathcal{V}_f(x)$ of the sequence of empirical measures is composed by all the (uncountably
infinitely many) $f-$invariant measures. Thus,
$\mathcal{V}_f(x)=\mathcal{O}_f=\mathcal{M}_f$ for Lebesgue almost
all $x \in S^1$. Then, it is easy to check  that there is no
strong observable measure in this example.

\begin{proposition}

\label{propositionMis} For the example \em \ref{exampleMis} \em of
Misiurewicz, there exist two observable ergodic invariant measures
$\mu$ and $\nu$, and a  subadditive potential $\Phi=\{\phi_n\}_n$,
such that the following properties hold:
\begin{itemize}
\item[(i)] $\Phi^*(x):=\lim_{n\rightarrow\infty}\frac 1 n
\phi_n(x)<0$ for $\mu-$a.e. $x$.

\item[(ii)] The smallest growth rates $\widehat{\Phi}^*(x)<0$ for
all $x\in A_{\epsilon}(\mu)$ for all $\epsilon >0$ small enough.

\item[(iii)] The largest growth rates $\widetilde{\Phi}^*(x)>0$
for   all $x\in A_{\epsilon}(\nu)$ for all $\epsilon >0$ small
enough.
\end{itemize}
\end{proposition}
\begin{proof}

 The map $f$ of
Misiurewicz has a dense set of periodic orbits (see Theorem 3.4 of
\cite{mi}). Choose two of those periodic orbits, say $
\mathcal{O}_1$ and $ \mathcal{O}_2$ and a real continuous function
$g: S^1\to \mathbb{R}$ such that $g(x) =-1$ for all $x\in
\mathcal{O}_1$ and $g(x) = 1$ for all $x \in \mathcal{O}_2$.

Define $\phi_n(x) =\sum_{i=0}^{n-1}g(f^ix)$. Note that $\Phi=
\{\phi_n\}_{n \geq 1}$ is an additive potential, since $\phi_{n+
m} = \phi_n + \phi_m(f^n)$ for all natural numbers $n $ and $ m $.
Therefore $\{\phi_n\}_n$ and  $\{-\phi_n\}_n$ are also subadditive
potentials.

The measure $\mu$ supported on $\mathcal{O}_1$ and equally
distributed in all the points of $\mathcal{O}_1$ is ergodic. And
besides it is observable because all invariant measures are
observable for $f$. Furthermore, by construction $\frac 1 n
\phi_n(x)=-1$ for
 all $x\in \mathcal{O}_1$, so $\Phi^*(x)=-1<0$ for $\mu-$a.e. $x$,
 proving (i). Therefore, applying theorem \ref{theoremB}, the assertion (ii) follows for some $\epsilon >0$, and after the definition of basin $A_{\epsilon} (\mu)$ of $\epsilon$- attraction, the assertion (ii) is proved
 for all $\epsilon>0$ small enough.

 Now it is left to prove (iii). Consider the measure $\nu$
 supported on $\mathcal{O}_2$ and equally distributed in all the points of $\mathcal{O}_2$. Similar arguments
 to those used with $\mu$ lead to the conclusion that $\nu$ is observable ergodic
 and $  {\Phi}^*(x)=+1$ for all $x\in \mathcal{O}_2$. Besides,  for all $x \in \mathcal{O}_2$, $- {\Phi}^*(x) = \Psi^*(x)$, where $\Psi := \{-\phi_n\}_{n \geq 1}$.
Therefore, ${\Psi}^*(x) = -1$ for $\nu$-a.e. point. Applying again
theorem \ref{theoremB} we obtain that $\widehat \Psi^*(x) < 0$ for
all $x \in A_{\epsilon}(\nu)$, for all $\epsilon >0$ small enough.
Since $\widehat \Psi^*(x) = - \widetilde \Phi^*(x)$, we conclude
that $\widetilde \Phi^*(x) >0$ for all $x \in A_{\epsilon}(\nu)$
for all $\epsilon >0$ small enough.   This
 completes the proof of (iii).
 \end{proof}
\end{example}

\begin{corollary}

  In the   example \em \ref{exampleMis} \em of Misiurewicz, if $\Phi$ is the  subadditive potentials   of Proposition \em \ref{propositionMis}, \em  then
$\Phi^*(x) < 0$ for $\mu$-a.e., but for Lebesgue almost all points
$x \in S^1$:
$$\widehat \Phi^*(x) < 0, \ \ \ \ \widetilde \Phi^*(x) >0.$$
\end{corollary}
\begin{proof}
Since $\mathcal{V}_f(x)=\mathcal{M}_f$ for Lebesgue almost all
point $x \in
 S^1$, $\mathcal{V}_f(x)$ intersects $A_{\epsilon}(\mu)$ for all $\epsilon>0$
 for Lebesgue almost all $x$ and for all $\mu \in \mathcal M_f$. This implies that for all positive
 $\epsilon$ and for all pair of invariant measures $\mu$ and $\nu$,
 $A_{\epsilon}(\mu)=A_{\epsilon }(\nu) = S^1$ up to   sets of zero
 Lebesgue measures. Hence, the assertion (ii) of Proposition \ref{propositionMis}
 implies that   $\widehat{\Phi}^*(x)<0$ for  Lebesgue almost  all $x\in
S_1$. Analogously the   assertion (iii) implies that $\widetilde
\Phi^*(x) >0$  for Lebesgue almost all $x \in S_1$.
\end{proof}
\begin{remark} \em
After the Corollary above, the example \ref{exampleMis} of
Misiurewicz shows that  the assertion $\Phi^*(x) <0$ for $\mu$
a.e., assumed in the hypothesis of Theorem \ref{theoremB} and
Corollary \ref{CorollaryB} can be satisfied,  but the  conclusions
of those two results stating that the smallest growth rate is
negative, can not be improved, since the largest growth rates may
be positive, as in this concrete example,
 for Lebesgue almost all $x$ in the  manifold.

Therefore, the last example shows that the conclusion of Theorem
\ref{theoremB} can not be strengthened
 in general. In this sense  Theorem
\ref{theoremB} and Corollary \ref{CorollaryB} are optimally stated
if one wishes them to hold for all the continuous systems.
\end{remark}

\section*{Acknowledgments}Authors are grateful to  anonymous referees for their comments, which
helped to improve the text. The first author thanks ANII and CSIC of the Universidad de la Rep\'{u}blica, Uruguay, for the partial financial support. The second author also would like to
take this opportunity to thank ICTP and NSFC for giving him the
financial support to attend the activity-School and Conference on
Computational Methods in Dynamics held in ICTP.


\end{document}